\documentclass[a4paper,11pt]{article}
\usepackage{amsmath}
\usepackage{amsfonts}
\usepackage{amssymb}
\usepackage{amsthm}
\usepackage{graphicx}

\date{} 
\begin{document} 
\centerline {$\displaystyle \delta$-PRIMARY ELEMENTS IN LATTICE MODULES} 
\centerline{} 

\centerline{\bf {A.V.Bingi$^{1}$, C.S.Manjarekar$^{2}$}}
\centerline{} 

\centerline{} 

\centerline{$^{1}$ Department of Mathematics}
\centerline{St.~Xavier's College(autonomous), Mumbai-400001, India}
\centerline{$email: ashok.bingi@xaviers.edu$}

\centerline{}
\centerline{$^{2}$ Formerly Professor at Department of Mathematics}
\centerline{Shivaji University, Kolhapur-416004, India}
\centerline{$email: csmanjrekar@yahoo.co.in$} 

\newtheorem{thm}{Theorem}[section]
 \newtheorem{c1}[thm]{Corollary}
 \newtheorem{l1}[thm]{Lemma}
 \newtheorem{prop}[thm]{Proposition}
 \newtheorem{d1}[thm]{Definition}
\newtheorem{rem}[thm]{Remark}
 \newtheorem{e1}[thm]{Example}
\begin{abstract}
  In this paper, we introduce the expansion function $\delta$ on an $L$-module $M$. We define and investigate a $\delta$-primary element in an $L$-module $M$. Its characterizations and many of its properties are obtained. $\delta_0$-primary and $\delta_1$-primary elements of an $L$-module $M$ are related with 2-absorbing, 2-absorbing primary elements of an $L$-module $M$ to obtain their special properties. The element $\delta_1(N)\in M$ is related to $rad(N)\in M$, the radical element of $M$ to obtain its properties where $N\in M$. We define a $\delta_L$-primary element in an $L$-module $M$ where $\delta_L$ is an expansion function on $L$ and find relation among a $\delta_L$-primary element of $M$ and its corresponding $\delta_L$-primary element of $L$.
\end{abstract}

\paragraph*{}
\section{Introduction}
\paragraph*{}  The study of expansions of ideals with $\delta$-primary ideals for commutative rings and semirings is carried out by D.~Zhao in  \cite{Z} and S.~Atani~et.~al. in \cite{ASK} respectively.  Further work on $\delta$-primary co-ideals of a commutative semiring and on $\delta$-primary submodules of modules is carried out by S.~Atani~et.~al. in \cite{APK} and G.~Yesilot~et.~al. in \cite{YSUU} respectively. In multiplicative lattices, the study of $\delta$-primary elements is done by C.~S.~Manjarekar and A.~V.~Bingi in \cite{MB}. Our aim is to extend the notion of $\delta$-primary elements in a multiplicative lattice to the notion of  $\delta$-primary elements in a lattice module and study its properties.

A multiplicative lattice $L$ is a complete lattice provided with commutative, associative and join distributive multiplication in which the largest element $1$ acts as a multiplicative identity. An element $e\in L$ is called meet principal if $a\wedge be=((a:e)\wedge b)e$ for all $a,b\in L$. An element $e\in L$ is called join principal if $(ae\vee b):e=(b:e)\vee a$ for all $a,b\in L$. An element $e\in L$ is called principal if $e$ is both meet principal and join principal. An element $a\in L$ is called compact if for  $X\subseteq L$, $a\leqslant \vee X$  implies the existence of a finite number of elements $a_1,a_2,\cdot\cdot\cdot,a_n$ in $X$ such that $a\leqslant a_1\vee a_2\vee\cdot\cdot\cdot\vee a_n$. The set of compact elements of $L$ will be denoted by $L_\ast$. If each element of $L$ is a join of compact elements of $L$ then $L$ is called a compactly generated lattice or simply a CG-lattice.  $L$ is said to be a principally generated lattice or simply a PG-lattice if each element of $L$ is the join of principal elements of $L$. Throughout this paper, $L$ denotes a compactly generated multiplicative lattice with $1$ compact in which every finite product of compact elements is compact.
       
An element $a\in L$ is said to be proper if $a<1$. A proper element $p\in L$ is called a prime element if $ab\leqslant p$ implies $a\leqslant p$ or $b\leqslant p$ where $a,b\in L$ and is called a primary element if $ab\leqslant p$ implies $a\leqslant p$ or $b^n\leqslant p$ for some $n\in Z_+$ where $a,b\in L_\ast$. For $a,b\in L $, $(a:b)= \vee \{x \in L \mid xb \leqslant a\} $. The radical of $a\in L$ is denoted by $\sqrt{a}$ and is defined as $\vee \{ x \in L_\ast \mid x^{n} \leqslant a$, for  some  $n\in Z_+\}$. According to \cite{JTY}, a proper element $q\in L$ is said to be 2-absorbing if for all $a, b, c\in L$,  $abc\leqslant q$ implies either $ab\leqslant q$ or $bc\leqslant q$ or $ca\leqslant q$. According to \cite{CYT}, a proper element $q\in L$ is said to be 2-absorbing primary if for all $a, b, c\in L$,  $abc\leqslant q$ implies either $ab\leqslant q$ or $bc\leqslant \sqrt{q}$ or $ca\leqslant \sqrt{q}$. The reader is referred to \cite{AAJ}, \cite{CYT} and \cite{JTY} for general background and terminology in multiplicative lattices.	
       	   
           Let $M$ be a complete lattice and $L$ be a multiplicative lattice. Then $M$ is called $L$-module or module over $L$ if there is a multiplication between elements of $L$ and $M$ written as $aB$ where $a \in L$  and $B \in M$ which satisfies the following properties: \\
           \textcircled{1}~~$(\underset{\alpha}{\vee} a_\alpha)A=\underset{\alpha}{\vee}(a_\alpha\ A)$  \\
           \textcircled{2}~~$a(\underset{\alpha}{\vee} A_\alpha)=\underset{\alpha}{\vee} (a\ A_\alpha)$  \\
           \textcircled{3}~~$(ab)A=a(bA)$  \\
           \textcircled{4}~~$1A=A$  \\
           \textcircled{5}~~$0A=O_M,  \ for \ all \ a, a_\alpha\ ,b \in L \ and \ A, A_\alpha \in M $ where $1$ is the supremum of $L$ and $0$ is  the infimum of $L$. We denote by $O_M$ and $I_M$ for the least element and the greatest element of $M$ respectively. Elements of $L$ will generally be denoted by $a,b,c,\cdot\cdot\cdot $ and elements of $M$ will generally be denoted by $A,B,C,\cdot\cdot\cdot$   
           
           Let $M$ be an $L$-module. For $N \in M$ and $a \in L$ , $(N:a) = \vee \{X \in M  \ \mid \  aX \leqslant N \}$. For $A,B \in M$, $(A:B) = \vee \{ x \in L  \ \mid  \ xB \leqslant A \} $. If $(O_M:I_M)=0$ then $M$ is called a faithful $L$-module. An $L$-module $M$ is called a multiplication lattice module if for every element $N \in M$ there exists an element $a \in L$ such that $N = aI_M$. An element $N\in M$ is called meet principal if $(b\wedge (B:N))N=bN\wedge B$ for all $b\in L, B\in M$.  An element $N\in M$ is called join principal if $b\vee (B:N)=((bN\vee B):N)$ for all $b\in L, B\in M$. An element $N\in M$ is said to be principal if $N$ is both meet principal and join principal. $M$ is said to be a PG-lattice $L$-module if each element of $M$ is the join of principal elements of $M$. An element $N\in M$ is called compact if $N\leqslant \underset{\alpha}{\vee} A_{\alpha}$ implies $N\leqslant A_{\alpha_{1}}\vee A_{\alpha_{2}}\vee\cdot\cdot\cdot\vee A_{\alpha_{n}}$ for some finite subset $\{\alpha_1,\alpha_2,\cdot\cdot\cdot,\alpha_n\}$. The set of compact elements of $M$ is denoted by $M_\ast$. If each element of $M$ is a join of compact elements of $M$ then $M$ is called a CG-lattice $L$-module.  
       An element $N\in M$ is said to be proper if $N < I_M$. A proper element $N\in M$ is said to be maximal if whenever there exists an element $B\in M$ such that $N \leqslant B$ then either $N=B$ or $B=I_M$. A proper element $N\in M$ is said to be prime if for all $a\in L$,  $X\in M$, $aX \leqslant N$ implies either $X\leqslant N$ or $aI_M \leqslant N$. A proper element $N\in M$ is said to be primary if for all $a\in L$, $X\in M$, $aX \leqslant N$ implies either $X\leqslant N$ or $a^n I_M \leqslant N$ for some $n\in Z_+$.  A proper element $N\in M$ is said to be a  radical element if $(N:I_M)=\sqrt{(N:I_M)}$ where $\sqrt{(N:I_M)}= \vee \{x\in L_\ast \mid x^n\leqslant (N:I_M) \ for \ some \ n\in Z_+\}$.  A proper element $N\in M$ is said to be semiprime if for all $a,\ b\in L$, $abI_M\leqslant N$ implies either $aI_M\leqslant N$ or $bI_M\leqslant N$. A proper element $N\in M$ is said to be $p$-prime if $N$ is prime and $p=(N:I_M)\in L$ is prime. A proper element $N\in M$ is said to be $p$-primary if $N$ is primary and $p=\sqrt{N:I_M}\in L$ is prime. A prime element $N\in M$ is said to be  minimal prime over $X\in M$ if $X\leqslant N$ and whenever there exists a prime element $Q\in M$ such that $X\leqslant Q\leqslant N$ then $Q=N$.  According to \cite{MB1}, a proper element $Q$ of an $L$-module $M$ is said to be 2-absorbing if for all $a, b\in L$, $N\in M$, $abN\leqslant Q$ implies either $ab\leqslant (Q:I_M)$ or $bN\leqslant Q$ or $aN\leqslant Q$.   According to \cite{BK}, a proper element $Q$ of an $L$-module $M$ is said to be 2-absorbing primary if for all $a, b\in L$, $N\in M$, $abN\leqslant Q$ implies either $ab\leqslant (Q:I_M)$ or $bN\leqslant (\sqrt{Q:I_M})I_M$ or $aN\leqslant (\sqrt{Q:I_M})I_M$. The reader is referred to \cite{A}, \cite{CT}, \cite{CTUO}, \cite{J} and \cite{MK1} for general background and terminology in lattice modules.
      
       According to \cite{MB},  an expansion function on a multiplicative lattice $L$ is a function $\delta:L \longrightarrow L $ which satisfies the following two conditions: \textcircled{1} $a\leqslant\delta(a)$ for all $a\in L$, \textcircled{2} $a\leqslant b$ implies $\delta(a)\leqslant\delta(b)$ for all $a,b\in L$ and further given an expansion function $\delta$ on $L$, an element $p\in L$ is called $\delta$-primary if for $a,b\in L$, $ab\leqslant p$ implies either $a\leqslant p$ or $b\leqslant\delta(p)$.
       
         This paper is motivated by \cite{MB},\cite{YSUU} and \cite{Z}. It should be mentioned that apart from new results obtained, there is a significant difference between some results obtained in this paper and the already existing ones in \cite{MB} because principal elements of $M$ are used wherever needed and some more conditions have to be imposed on $M$. The most important outcome of this paper is the result $({\underset{\alpha}{\wedge}a_\alpha}) I_M=\underset{\alpha}{\wedge}(a_\alpha I_M)$ which is presented as Lemma \ref{L-C91} and this result is somewhat close to the dual of property \textcircled{1} in the definition of an $L$-module $M$. We define a meet prime element in an $L$-module $M$ and prove that every maximal element in a multiplication $L$-module $M$ is meet prime. This result is presented as Lemma \ref{L-C92}. 
                
                In the first section of this paper, we introduce  expansion function on an $L$-module $M$ and study $\delta$-primary elements of $M$ which unify prime and primary elements of $M$. In the second section, we define a special property of expansion function on an $L$-module $M$. Then we obtain special properties of $\delta_0$-primary and $\delta_1$-primary elements in an $L$-module $M$ by relating them with 2-absorbing and 2-absorbing primary elements of $M$. Also, some properties of the element $\delta_1(N)\in M$ are found by relating it with $rad(N)\in M$ where $N\in M$. In the last section, given an expansion function $\delta_L$ on $L$, a $\delta_L$-primary element of an $L$-module $M$ is defined and its characterizations are obtained. Finally, with respect to expansion function $\delta$ on $L$, we find the interrelation between an element in an $L$-module $M$ and its corresponding element in $L$. 
                
               We denote expansion function on a multiplicative lattice $L$ by  $\delta_L$ instead of $\delta$ wherever needed to distinguish it from expansion function $\delta$ on an $L$-module $M$.

\section{Expansion function $\delta$ on $M$ and $\delta$-primary element of $M$}
       
       \paragraph*{} We begin with introducing the notion of an expansion function on an $L$-module $M$.
           
       \begin{d1} 
      An {\bf expansion function} on an $L$-module $M$ is a function $\delta:M \longrightarrow M$ which satisfies the following two conditions:
      
      \textcircled{1}. $A\leqslant\delta(A)$ for all $A\in M$.
      
      \textcircled{2}. $A\leqslant B$ implies $\delta(A)\leqslant\delta(B)$ for all $A,\ B\in M$. 
      \end{d1}
       
      The reader can verify that the following functions on an $L$-module $M$ are expansion functions:
      \begin{enumerate}
      \item the identity function $\delta_0:M\longrightarrow M$ defined as $\delta_0(A)= A$ for all $A\in M$.      
      \item the function $\delta_1:M\longrightarrow M$ defined as $\delta_1(A)=(\sqrt{A:I_M})I_M$ for all $A\in M$ provided $M$ is a multiplication lattice module.
      \item the function $\delta_2:M\longrightarrow M$ defined as $\delta_2(A)$=$\wedge\{$ $H\in M \mid A\leqslant H$ and $H$  is a maximal element of $M$  $\}$ for all proper elements $A\in M$ and $\delta_2(I_M)=I_M$.
      \end{enumerate}  
                 
        The following result shows that the arbitrary meet of any collection of expansion functions on an $L$-module $M$ is again an expansion function on $M$.
              
              \begin{thm}
              Given any two expansion functions $\gamma_1$ and $\gamma_2$ on an $L$-module $M$, define a function $\delta:M \longrightarrow M$ as $\delta(A)=\gamma_1(A)\wedge\gamma_2(A)$ for all $A\in M$. Then $\delta$ is also an expansion function on $M$.
              \end{thm}              
              \begin{proof}
             Let $A\in M$. Then as $\gamma_1$ and $\gamma_2$ are expansion functions on $M$, we have $A\leqslant\gamma_1(A)$ and  $A\leqslant\gamma_2(A)$ which implies $A\leqslant\gamma_1(A)\wedge\gamma_2(A)=\delta(A)$. Now for $A,\ B\in M$, let $A\leqslant B$. Then as $\gamma_1$ and $\gamma_2$ are expansion functions on $M$, we have $\gamma_1(A)\leqslant\gamma_1(B)$ and  $\gamma_2(A)\leqslant\gamma_2(B)$. So $\gamma_1(A)\wedge\gamma_2(A)\leqslant\gamma_1(B)\wedge\gamma_2(B)$ which implies $\delta(A)\leqslant\delta(B)$ and hence $\delta$ is an expansion function on $M$.
             \end{proof}
              
               We now introduce the notion of $\delta$-primary elements of an $L$-module $M$ which is the generalization of the concept of $\delta$-primary elements of $L$ studied in \cite{MB}.
              \begin{d1} 
               Given an expansion function $\delta$ on an $L$-module $M$, a proper element $P\in M$ is called {\bf $\delta$-primary} if for all $A\in M$ and $a\in L$, $aA\leqslant P$ implies either $A\leqslant P$ or $aI_M\leqslant\delta(P)$. 
              \end{d1}       
         The above definition is equivalent to the following:
         
         Given an expansion function $\delta$ on an $L$-module $M$, a proper element $P\in M$ is called {\bf $\delta$-primary} if for all $A\in M$ and $a\in L$, $aA\leqslant P$ implies either $A\leqslant\delta(P)$ or $aI_M\leqslant P$.
         
         The reader can verify that given an expansion function $\delta$ on an $L$-module $M$, the function  $E_\delta (A):M\longrightarrow M$ defined as $E_\delta (A)= \wedge \{$ $J\in M \mid A \leqslant J$ and $J$ is a $\delta$-primary element of $M$ $\}$ for all proper elements $A\in M$ with $E_\delta(I_M)=I_M$ is an expansion function on an $L$-module $M$.
         
        Now we give characterizations of a $\delta$-primary element of an $L$-module $M$.
                 
       \begin{thm}
       Let $M$ be a CG-lattice $L$-module, $N\in M$ be a proper element and $\delta$ be an expansion function on $M$. Then the following statements are equivalent:-
       \begin{enumerate}
        \item $N$ is a $\delta$-primary element of $M$.
        \item for every $r\in L$ such that $r\nleqslant (\delta(N):I_M)$, we have $(N:r)=N$.
        \item for every $r\in L_\ast$, $A\in M_\ast$ , if $rA\leqslant N$ then either $A\leqslant N$ or $rI_M\leqslant \delta(N)$.
       \end{enumerate} 
       \end{thm}
        \begin{proof}
          $(1)\Longrightarrow (2)$. Suppose $(1)$ holds. Let $Q\leqslant (N:r)$ be such that $r\nleqslant (\delta(N):I_M)$ for $Q\in M, r\in L$. Then $rQ\leqslant N$. As $N$ is a $\delta$-primary element of $M$ and $rI_M\nleqslant \delta(N)$, we have, $Q\leqslant N$ and thus $(N:r)\leqslant N$. Since $rN\leqslant N$ gives $N\leqslant (N:r)$, it follows that $(N:r)=N$.\\
          $(2)\Longrightarrow (3)$. Suppose $(2)$ holds. Let $rA\leqslant N$ and  $r\nleqslant (\delta(N):I_M)$ for  $r\in L_\ast$, $A\in M_\ast$. Then by $(2)$, we have, $(N:r)=N$. So $rA\leqslant N$ gives $A\leqslant(N:r)=N$.\\
          $(3)\Longrightarrow (1)$. Suppose $(3)$ holds. Let $aQ\leqslant N$ and $Q\nleqslant N$ for $Q\in M, a\in L$.  As $L$ and $M$ are compactly generated, there exist $x'\in L_\ast$ and $Y,Y'\in M_\ast$ such that $x'\leqslant a, Y\leqslant Q, Y'\leqslant Q$ and $Y'\nleqslant N$. Let $x\in L_\ast$ such that $x\leqslant a$. Then $(x\vee x')\in L_\ast$, $(Y\vee Y')\in M_\ast$ such that $(x\vee x')(Y\vee Y')\leqslant aQ\leqslant N$ and $(Y\vee Y')\nleqslant N$. So by $(3)$, $(x\vee x')I_M\leqslant \delta(N)$ which implies $x\leqslant (\delta(N):I_M)$. Thus $a\leqslant (\delta(N):I_M)$ and hence $aI_M\leqslant \delta(N)$. Therefore $N$ is a $\delta$-primary element of $M$.      
          \end{proof}
                 
          \begin{thm}
          Let $M$ be a CG-lattice $L$-module, $N\in M$ be a proper element and $\delta$ be an expansion function on $M$. Then the following statements are equivalent:-
          \begin{enumerate}
          \item $N$ is a $\delta$-primary element of $M$.
          \item for every $A\in M$ such that $A\nleqslant N$, we have $(N:A)\leqslant (\delta(N):I_M)$.
          \item for every $r\in L_\ast$, $A\in M_\ast$ , if $rA\leqslant N$ then either $A\leqslant N$ or $rI_M\leqslant \delta(N)$.
          \end{enumerate} 
           \end{thm}
          \begin{proof}
          $(1)\Longrightarrow (2)$. Suppose $(1)$ holds. Let $A\in M$ be such that $A\nleqslant N$. Let $a\in L_\ast$ be such that $a\leqslant (N:A)$. Then $aA\leqslant N$. As $N$ is a $\delta$-primary element of $M$ and  $A\nleqslant N$, we have, $aI_M\leqslant \delta(N)$ which implies $a\leqslant (\delta(N):I_M)$  and thus $(N:A)\leqslant (\delta(N):I_M)$.\\
          $(2)\Longrightarrow (3)$. Suppose $(2)$ holds. Let $rA\in N$ and $A\nleqslant N$ for $r\in L_\ast$, $A\in M_\ast$. Then by $(2)$, we have,  $(N:A)\leqslant (\delta(N):I_M)$. So $rA\leqslant N$ gives $r\leqslant (N:A)\leqslant (\delta(N):I_M)$ which implies $rI_M\leqslant \delta(N)$.\\
          $(3)\Longrightarrow (1)$. Suppose $(3)$ holds. Let $aQ\leqslant N$ and $aI_M\nleqslant \delta(N)$ for $Q\in M, a\in L$. Then $a\nleqslant (\delta(N):I_M)$.  As $L$ and $M$ are compactly generated, there exist $x', x\in L_\ast$ and $Y\in M_\ast$ such that $x\leqslant a, x'\leqslant a, Y'\leqslant Q$ and $x'\nleqslant (\delta(N):I_M)$. Let $Y\in M_\ast$ such that $Y\leqslant Q$. Then $(x\vee x')\in L_\ast$, $(Y\vee Y')\in M_\ast$ such that $(x\vee x')(Y\vee Y')\leqslant aQ\leqslant N$ and $(x\vee x')\nleqslant (\delta(N):I_M)$, i.e.$(x\vee x')I_M\nleqslant (\delta(N)$. So by $(3)$, $(Y\vee Y')\leqslant N$ which implies $Y\leqslant N$. Thus $Q\leqslant N$ and hence $N$ is a $\delta$-primary element of $M$.  
          \end{proof} 
         
        The following two results characterize the $\delta_0$-primary and $\delta_1$-primary elements of an $L$-module $M$ and relate them with prime and primary elements of $M$. 
         
          \begin{thm}\label{T-C41}
          Given an expansion function $\delta_0$ on an $L$-module $M$, a proper element $P$ of an $L$-module $M$ is $\delta_0$-primary if and only if it is prime.
          \end{thm}
          \begin{proof}
          The proof is obvious.
           \end{proof}
           
           \begin{thm}\label{T-C4}
           Let $L$ be a PG-lattice and $M$ be a faithful multiplication PG-lattice $L$-module with $I_M$ compact. Then given an expansion function $\delta_1$ on an $L$-module $M$, a proper element $P\in M$ is $\delta_1$-primary if and only if it is primary.
            \end{thm}
             \begin{proof}
              Assume that a proper element $P\in M$ is primary. Let $aA\leqslant P$ and $A\nleqslant P$ for $a\in L,\ A\in M$. Then as $P$ is primary, we have $a\leqslant \sqrt{P:I_M}$. So $aI_M\leqslant (\sqrt{P:I_M})I_M=\delta_1(P)$ and hence $P$ is a $\delta_1$-primary element of $M$. Conversely, assume that $P$ is a $\delta_1$-primary element of $M$. Let $aA\leqslant P$ and $A\nleqslant P$ for $a\in L,\ A\in M$. Then as $P$ is  $\delta_1$-primary, we have $aI_M\leqslant \delta_1(P)=(\sqrt{P:I_M})I_M$. Since $I_M$ is compact, by Theorem 5 of \cite{CT}, it follows that $a\leqslant \sqrt{P:I_M}$ and hence $P$ is primary.  
              \end{proof}
              
              \begin{c1}\label{C-C01}
             Every primary element of a multiplication $L$-module $M$ is $\delta_1$-primary where $\delta_1$ is an expansion function on $M$.
              \end{c1}              
               \begin{proof}
                     Refer first part of the proof of Theorem \ref{T-C4}.
                     \end{proof}

      Clearly, every $\delta_0$-primary element of a multiplication lattice $L$-module $M$ is $\delta_1$-primary. But the converse is not true which is clear from the following example.  
      
      \begin{e1}\label{E-C1}
      Consider the lattice $L$ of ideals of the ring $R=<Z_{12}\ , \ +\ ,\ \cdot>$. Then the only ideals of $R$ are the principal ideals (0),(1),(2),(3),(4),(6). Clearly, $L=\{(0),(1),(2),(3),(4),(6)\}$
      is a compactly generated multiplicative lattice and $L$ is a module over itself. Also, $L$ is a multiplication lattice module.
      It is easy to see that the element $(4)$ is $\delta_1$-primary but not $\delta_0$-primary.
      \end{e1} 
      
      
      We note that for radical elements of a multiplication lattice $L$-module $M$, every $\delta_1$-primary element is $\delta_0$-primary. Also, in a multiplication lattice $L$-module $M$, every semiprime element is $\delta_0$-primary.
      
      Now we define a semiprimary element in an $L$-module $M$.
      
      \begin{d1}
       A proper element $N\in M$ is said to be {\bf semiprimary} if $\sqrt{N:I_M}$ is a prime element of $L$.
       \end{d1}
      
      According to Proposition 3.6 in \cite{A}, if $N$ is a prime element of an $L$-module $M$ then $(N:I_M)$ is a prime element of $L$. By Theorem 1 in \cite{MK2}, if Q is a primary element of an $L$-module $M$ then $\sqrt{Q:I_M}$ is a prime element of $L$.  By Theorem 2.2 in \cite{MB}, an element $p\in L$ is $\delta_0$-primary if and only if it is prime. In view of these 3 results and Theorem \ref{T-C41} along with Corollary \ref{C-C01}, we have the following corollary which gives the link between the $\delta$-primary elements  of an $L$-module $M$ and the corresponding $\delta_L$-primary elements of $L$. 
      
      \begin{c1}
      If $A$ is a semiprimary element of an $L$-module $M$ then $\sqrt{A:I_M}$ is a $(\delta_0)_L$-primary element of $L$.  If $Q$ is a $p$-prime element of an $L$-module $M$ then $Q\in M$ is $\delta_0$-primary and $p=(Q:I_M)\in L$ is  $(\delta_0)_L$-primary. Further, if $Q$ is a $p$-primary element of a multiplication $L$-module $M$ then $Q\in M$ is $\delta_1$-primary and $p=\sqrt{Q:I_M}\in L$ is  $(\delta_0)_L$-primary. 
      \end{c1}
      \begin{proof}
      The proof is obvious.
      \end{proof}      
          
      The next theorem shows that a prime element of an $L$-module $M$ is $\delta$-primary.        
                    
       \begin{thm}\label{T-C2}
      If $\delta$ and $\gamma$ are expansion functions on an $L$-module $M$ such that $\delta(A)\leqslant\gamma(A)$ for all $A\in M$ then every $\delta$-primary element of $M$ is $\gamma$-primary. In particular, a prime element of $M$ is $\delta$-primary for every expansion function $\delta$ on $M$.
      \end{thm}
       \begin{proof}
       Let $P\in M$ be $\delta$-primary and let $aA\leqslant P$ for $a\in L,\ A\in M$. Then either $A\leqslant P$ or $aI_M\leqslant \delta(P)\leqslant \gamma(P)$ which implies $P$ is $\gamma$-primary. Next, let $P\in M$ be a prime element. Then by Theorem $\ref{T-C41}$, $P$ is $\delta_0$-primary. For any expansion function $\delta$ on $M$, we have $P\leqslant\delta(P)$ and so $\delta_0(P)\leqslant\delta(P)$. Thus a prime element $P\in M$ is $\delta$-primary for every expansion function $\delta$ on $M$. 
       \end{proof}
       
       Now we show that $\delta_1(P)\leqslant \delta(P)$ for every $\delta$-primary element $P\in M$.
       
       \begin{thm}
       Let $L$ be a PG-lattice and M be a faithful multiplication PG-lattice $L$-module with $I_M$ compact. If $\delta$ is an expansion function on M such that $\delta(A)\leqslant\delta_1(A)$ for all $A\in M$ then for any $\delta$-primary element $P\in M$, $\delta(P)=\delta_1(P)$.
       \end{thm}
       \begin{proof}
       Assume that a proper element $P\in M$ is $\delta$-primary. For $A\in M$, let $A\leqslant\delta_1(P)=(\sqrt{P:I_M})I_M$. Since $M$ is a multiplication lattice $L$-module, $A=aI_M$ for some $a\in L$ and so $aI_M\leqslant(\sqrt{P:I_M})I_M$. Then since $I_M$ is compact, by Theorem 5 of \cite{CT}, we have $a\leqslant\sqrt{P:I_M}$. It follows that there exists a least positive integer $k$ such that $a^{k}I_M\leqslant P$. If $k=1$ then $A=aI_M\leqslant P\leqslant \delta(P)$. So let $k>1$. Clearly, $a^{k-1}aI_M\leqslant P$ and $a^{k-1}I_M\nleqslant P$. Then as $P\in M$ is $\delta$-primary, we have $A=aI_M\leqslant\delta(P)$. Thus in any case, we have $A\leqslant \delta(P)$ and hence $\delta_1(P)\leqslant\delta(P)$. But by hypothesis, $\delta(P)\leqslant\delta_1(P)$ and therefore $\delta(P)=\delta_1(P)$. 
       \end{proof}
       
       The following theorem is an interesting property of a $\delta$-primary element of $M$.
       
       \begin{thm}\label{T-C90}
       Let $\delta$ be an expansion function on an $L$-module $M$ and $P\in M$ be $\delta$-primary. Then
       
        \textcircled{1} $(P:a)=P$ for all $a\in L$ such that $aI_M\nleqslant\delta(P)$.
        
        \textcircled{2} $(P:q)$ is a $\delta$-primary element of $M$ for all $q\in L$.
       \end{thm}
       \begin{proof}
       \textcircled{1}. As $P\in M$ is $\delta$-primary, $a(P:a)\leqslant P$ and $aI_M\nleqslant\delta(P)$, we have $(P:a)\leqslant P$. But $P\leqslant (P:a)$ and hence $(P:a)=P$.
       
       \textcircled{2}. Clearly, $qP\leqslant P$ implies $P\leqslant (P:q)$ and so $\delta(P)\leqslant\delta(P:q)$. Now let $aA\leqslant (P:q)$ but $A\nleqslant (P:q)$ for $a\in L,\ A\in M$. Then $a(qA)\leqslant P$ and $qA\nleqslant P$. As $P$ is $\delta$-primary, we have $aI_M\leqslant \delta(P)\leqslant\delta(P:q)$ and hence $(P:q)$ is a $\delta$-primary element of $M$. 
       \end{proof}
       
       We conclude this section by the following result which shows that given a chain of $\delta$-primary elements of an $L$-mdoule $M$, their supremum is also $\delta$-primary.
       
       \begin{thm}\label{T-C91}
       Let $\{P_i\mid i\in \triangle\}$ be a chain of $\delta$-primary elements of an $L$-module $M$ with $I_M$ compact where $\delta$ is an expansion function on an $L$-module $M$. Then the element $P=\underset{i\in \triangle}{\vee}P_i$ is also a $\delta$-primary element of $M$.
       \end{thm}
       \begin{proof}
       Since $I_M$ is compact, $\underset{i\in \triangle}{\vee}P_i=P\neq I_M$. Let $aA\leqslant P$ and $A\nleqslant P$ for $a\in L,\ A\in M$. Then $aA\leqslant P_i$ for some $i\in \triangle$ but $A\nleqslant P_i$. As each $P_i$ is $\delta$-primary, we get $aI_M\leqslant \delta(P_i)$. Since $P_i\leqslant P$, we have $\delta( P_i)\leqslant\delta(P)$ and so $aI_M\leqslant\delta(P)$. Hence $P$ is a $\delta$-primary element of $M$.
       \end{proof}
       
       In the next section, we deal with the infimum of $\delta$-primary elements of an $L$-module $M$. 
          
       \section{Special properties of $\delta_0$, $\delta_1$ and $\delta_0$-primary, $\delta_1$-primary elements of $M$}
       
       Now we define expansion function on an $L$-module $M$ with a special property namely; the meet preserving property.
       \begin{d1}
       An expansion function $\delta$ on an $L$-module $M$ is said to have {\bf meet preserving property} if $\delta(A\wedge B)=\delta(A)\wedge \delta(B)$ for all $A,\ B\in M$. 
       \end{d1}
       
        The following result which shows that given a finite number of $\delta$-primary elements of an $L$-module $M$ then their meet (infimum) is also $\delta$-primary if $\delta$ has the meet preserving property.
              
              \begin{thm}\label{T-C92}
              Let the expansion function $\delta$ on an $L$-module $M$ have the meet preserving property. If $Q_1,\ Q_2,\cdots,\ Q_n$ are $\delta$-primary elements of $M$ and $P=\delta(Q_i)$ for all $i=1,\ 2,\cdots,\ n$ then the element $Q=\bigwedge_{i=1}^{n}Q_i$ is also a $\delta$-primary element of $M$.
              \end{thm}
              \begin{proof}
              Clearly, $\bigwedge_{i=1}^{n}Q_i=Q\neq I_M$. Let $aA\leqslant Q$ and $A\nleqslant Q$ for $a\in L,\ A\in M$. Then $A\nleqslant Q_k$ for some $k\in \{1,\ 2,\cdots,\ n\}$ but $aA\leqslant Q_k$. As each $Q_k$ is a $\delta$-primary element of $M$, we get $aI_M\leqslant \delta(Q_k)$. Now the meet preserving property of $\delta$ on $M$ gives $\delta(Q)=\delta(\bigwedge_{i=1}^{n}Q_i)=\bigwedge_{i=1}^{n}\delta(Q_i)=P=\delta(Q_k)$ and so $aI_M\leqslant\delta(Q)$. Hence $Q$ is a $\delta$-primary element of $M$.
              \end{proof}
       
       Clearly, the expansion function $\delta_0$ on an $L$-module $M$ has the meet preserving property. To show that the expansion function $\delta_1$ on an $L$-module $M$ has the meet preserving property, we need the following lemma.
       
       \begin{l1}\label{L-C91}
       Let $L$ be a PG-lattice and $M$ be a faithful multiplication PG-lattice $L$-module. Then  $\underset{\alpha\in \triangle}{\wedge}(a_\alpha I_M)=({\underset{\alpha\in \triangle}{\wedge}a_\alpha}) I_M$ where $\{a_\alpha \in L \mid \alpha\in \triangle \}$.
       \end{l1}
       \begin{proof}
       Clearly, $({\underset{\alpha\in \triangle}{\wedge}a_\alpha}) I_M\leqslant \underset{\alpha\in \triangle}{\wedge}(a_\alpha I_M)$. Let $X\leqslant \underset{\alpha\in \triangle}{\wedge}(a_\alpha I_M)$ where $X\in M$. We may suppose that $X$ is a principal element. Assume that $(({\underset{\alpha\in \triangle}{\wedge}a_\alpha}) I_M:X)\neq 1$. Then there exists a maximal element $q\in L$ such that $(({\underset{\alpha\in \triangle}{\wedge}a_\alpha}) I_M:X)\leqslant q$. As $M$ is a multiplication lattice $L$-module and $q\in L$ is maximal, by Theorem 4 of \cite{CT}, two cases arise: 
       
       Case\textcircled{1}. For principal element $X\in M$, there exists a principal element $r\in L$ with $r\nleqslant q$ such that $r X=O_M$. Then $r\leqslant (O_M:X)\leqslant (({\underset{\alpha\in \triangle}{\wedge}a_\alpha}) I_M:X)\leqslant q$ which is a contradiction. 
       
       Case\textcircled{2}. There exists a principal element $Y\in M$ and a principal element $b\in L$ with $b\nleqslant q$ such that $bI_M\leqslant Y$. Then $bX\leqslant Y$, $bX\leqslant b[\underset{\alpha\in \triangle}{\wedge}(a_\alpha I_M)]\leqslant \underset{\alpha\in \triangle}{\wedge}(a_\alpha b I_M)\leqslant \underset{\alpha\in \triangle}{\wedge}(a_\alpha Y)$ and $(O_M:Y)bI_M\leqslant (O_M:Y)Y=O_M$, since $Y$ is meet principal. As $M$ is faithful, it follows that $b(O_M:Y)=0$. Since $Y$ is meet principal, $(bX:Y)Y=bX\wedge Y=bX$. Let $s=(bX:Y)$ then $sY=bX\leqslant \underset{\alpha\in \triangle}{\wedge}(a_\alpha Y)$. So $s=(bX:Y)=(sY:Y)\leqslant [\underset{\alpha\in \triangle}{\wedge}(a_\alpha Y):Y]=\underset{\alpha\in \triangle}{\wedge}(a_\alpha Y:Y)=\underset{\alpha\in \triangle}{\wedge}[a_\alpha \vee (O_M:Y)]$, since $Y$ is join principal. Therefore $bs\leqslant b[\underset{\alpha\in \triangle}{\wedge}[a_\alpha \vee (O_M:Y)]]\leqslant \underset{\alpha\in \triangle}{\wedge}[b[a_\alpha \vee (O_M:Y)]]=\underset{\alpha\in \triangle}{\wedge}[(ba_\alpha) \vee b(O_M:Y)]=\underset{\alpha\in \triangle}{\wedge}(ba_\alpha)\leqslant b\wedge (\underset{\alpha\in \triangle}{\wedge}a_\alpha)\leqslant (\underset{\alpha\in \triangle}{\wedge}a_\alpha)$ and so $b^2X=b(bX)=bsY\leqslant (\underset{\alpha\in \triangle}{\wedge}a_\alpha)Y\leqslant (\underset{\alpha\in \triangle}{\wedge}a_\alpha)I_M$. Hence $b^2\leqslant (({\underset{\alpha\in \triangle}{\wedge}a_\alpha}) I_M:X)\leqslant q$ which implies $b\leqslant \sqrt{q}=q$, a contradiction. 
       
       Thus the assumption that $(({\underset{\alpha\in \triangle}{\wedge}a_\alpha}) I_M:X)\neq 1$ is absurd and so we must have $(({\underset{\alpha\in \triangle}{\wedge}a_\alpha}) I_M:X)=1$ which implies $X\leqslant ({\underset{\alpha\in \triangle}{\wedge}a_\alpha})I_M$. It follows that  $\underset{\alpha\in \triangle}{\wedge}(a_\alpha I_M)\leqslant ({\underset{\alpha\in \triangle}{\wedge}a_\alpha}) I_M$ and hence $\underset{\alpha\in \triangle}{\wedge}(a_\alpha I_M)=({\underset{\alpha\in \triangle}{\wedge}a_\alpha}) I_M$.
       \end{proof}
         
       \begin{thm}
       Let $L$ be a PG-lattice and $M$ be a faithful multiplication PG-lattice $L$-module. The expansion function $\delta_1$ on an $L$-module $M$ has the meet preserving property.
       \end{thm}
       \begin{proof}
       Let $A,\ B\in M$. Now using property (2.17) of residuation in \cite{J}, we have  $\delta_1(A\wedge B)= (\sqrt{(A\wedge B):I_M})I_M=(\sqrt{(A:I_M)\wedge (B:I_M)})I_M$. But by property (p4) of radicals in \cite{TM}, we have $(\sqrt{(A:I_M)\wedge (B:I_M)})I_M=(\sqrt{A:I_M}\wedge \sqrt{B:I_M})I_M$.   Hence by Lemma $\ref{L-C91}$, $\delta_1(A\wedge B)=(\sqrt{A:I_M})I_M\wedge (\sqrt{B:I_M})I_M=\delta_1(A)\wedge \delta_1(B)$. 
       \end{proof}
       
       Now we define a meet prime element in an $L$-module $M$.
       
       \begin{d1}
       A proper element $N\in M$ is said to be {\bf meet prime} if for all $A,\ B\in M$, $A\wedge B\leqslant N$ implies either $A\leqslant N$ or $B\leqslant N$.
       \end{d1}
       
       To show that the expansion function $\delta_2$ on an $L$-module $M$ has the meet preserving property, we need the following lemma.
       
       \begin{l1}\label{L-C92}
       Let $L$-module $M$ be a multiplication module. Then every maximal element in $M$ is meet prime.
       \end{l1}
       \begin{proof}
       Let $H\in M$ be a maximal element and let $A\wedge B\leqslant H$  for $A, \ B\in M$. Since $M$ is a multiplication lattice module, we have $A=aI_M$ and $B=bI_M$ for some $a,\ b\in L$. Then $abI_M\leqslant aI_M \wedge bI_M \leqslant H$ where $H$ being maximal is prime, by Corollary 3.1 of \cite{A}. But then by Lemma 3.1 of \cite{A}, $H$ is semiprime which implies either $aI_M=A\leqslant H$ or $bI_M=B\leqslant H$ and hence $H$ is a meet prime element of $M$.
       \end{proof}
       
       \begin{thm}
       Let $M$ be a multiplication $L$-module. The expansion function $\delta_2$ on an $L$-module $M$ has the meet preserving property.
       \end{thm}
       \begin{proof}
       Let $A,\ B\in M$. Let $H_1$=$\{$$H\in M \mid (A\wedge B) \leqslant H,\ H \ is$  \ a maximal element$\}$ and $H_2$=$\{$$H\in M \mid A\leqslant H \ or \ B\leqslant H,\ H \ is$  \ a maximal element$\}$ where $A,\ B\in M $ are proper. Then $\wedge H_1=\delta_2(A\wedge B)$. Now $\delta_2(A)\wedge \delta_2(B)=(\wedge\{$$H\in M \mid A\leqslant H,\ H \ is$  \ a maximal element$\}$)$\wedge$ ($\wedge\{$$H\in M \mid B\leqslant H,\ H \ is$  \ a maximal element$\}$)=$\wedge$$\{$$H\in M \mid A\leqslant H \ or \ B\leqslant H,\ H \ is$  \ a maximal element$\}$=$\wedge H_2$. If $H\in H_2$ then either $A\leqslant H$ or $B\leqslant H$. So $A\wedge B\leqslant A\leqslant H$ or $A\wedge B\leqslant B\leqslant H$. Thus $H\in H_1$ which implies $H_2\subseteq H_1$. If $H\in H_1$ then $A\wedge B\leqslant H$.  This implies that either $A\leqslant H$ or $B\leqslant H$, because by Lemma $\ref{L-C92}$,  $H$ being a maximal element of $M$,  $H$ is a meet prime element of $M$. Thus $H\in H_2$ which implies $H_1\subseteq H_2$. Hence $\delta_2(A\wedge B)=\delta_2(A)\wedge \delta_2(B)$ for all $A,\ B\in M$.
       \end{proof}

       Now we will see {\bf special properties of $\delta_0$-primary and $\delta_1$-primary elements} in an $L$-module $M$ by relating it with 2-absorbing and 2-absorbing primary elements of an $L$-module $M$.
       
       The interrelations among prime, primary, 2-absorbing and 2-absorbing primary elements of $M$ are given in following theorems whose proofs being obvious are omitted. 
       
       \begin{thm}\label{T-C04}
       Every prime element of an $L$-module $M$ is primary and 2-absorbing.
       \end{thm}
       
       \begin{thm}\label{T-C05}
       If $Q$ is a primary element of an $L$-module $M$ then $\sqrt{Q:I_M}$ is a prime element and hence a 2-absorbing element of $L$. Also, it is a 2-absorbing primary element of $L$. 
       \end{thm}
                    
       \begin{thm}\label{T-C06}
       If $Q$ is a 2-absorbing element of an $L$-module $M$ then both, $\sqrt{Q:I_M}$ and $(Q:I_M)$ are 2-absorbing elements of $L$. Also, they are 2-absorbing primary elements of $L$.
       \end{thm}
       
       \begin{thm}\label{T-C01}
       Every 2-absorbing element of a multiplication $L$-module $M$ is 2-absorbing primary.
       \end{thm}
       
        \begin{thm}\label{T-C02}
        Every primary element of a multiplication $L$-module $M$ is 2-absorbing primary.
        \end{thm}
             
       \begin{thm}\label{T-C07}
       Let $L$ be a PG-lattice and M be a faithful multiplication PG-lattice $L$-module with $I_M$ compact. If $Q$ is a 2-absorbing primary element of $M$ then $(Q:I_M)$ is a 2-absorbing primary element of $L$ and $\sqrt{Q:I_M}$ is a 2-absorbing element of $L$.
       \end{thm}
        \begin{proof}
        Let $abc\leqslant (Q:I_M)$ for $a, b, c\in L$. Then as $ab(cI_M)\leqslant Q$ and $Q$ is a 2-absorbing primary  element of $M$, we have, either $ab\leqslant (Q:I_M)$ or $a(cI_M)\leqslant (\sqrt{Q:I_M})I_M$ or $b(cI_M)\leqslant (\sqrt{Q:I_M})I_M$. Since $I_M$ is compact, by Theorem 5 of \cite{CT}, it follows that, either $ab\leqslant (Q:I_M)$ or $ac\leqslant \sqrt{Q:I_M}$ or $bc\leqslant \sqrt{Q:I_M}$ and hence $(Q:I_M)$ is a 2-absorbing primary element of $L$. By Theorem 2.4 in \cite{CYT},it follows that $\sqrt{Q:I_M}$ is a 2-absorbing element of $L$.
       \end{proof}
       
       Relating this concepts to the Theorems \ref{T-C41}, \ref{T-C4} and Corollary \ref{C-C01}, we get, following properties of $\delta_0$-primary and $\delta_1$-primary elements of an $L$-module $M$.       
       
      \begin{thm}\label{T-C08}
      Every $\delta_0$-primary element $Q$ of an $L$-module $M$ is primary and 2-absorbing where $\delta_0$ is an expansion function on $M$. Also then both, $\sqrt{Q:I_M}$ and $(Q:I_M)$ are 2-absorbing and hence 2-absorbing primary elements of $L$. 
      \end{thm}  
      \begin{proof}
      The proof follows from Theorems \ref{T-C41}, \ref{T-C04} and \ref{T-C06}.
      \end{proof}
              
         \begin{thm}\label{T-C09}
         Every $\delta_0$-primary element $Q$ of a multiplication $L$-module $M$ is 2-absorbing primary where $\delta_0$ is an expansion function on $M$. Also then $(Q:I_M)$ is a 2-absorbing primary element of $L$ provided $M$ is a faithful multiplication PG-lattice with $I_M$ compact and $L$ as a PG-lattice. Further, $\sqrt{Q:I_M}$ is a 2-absorbing element of $L$.
        \end{thm}  
        \begin{proof}
        The proof follows from Theorems \ref{T-C41}, \ref{T-C04}, \ref{T-C02} and \ref{T-C07}.
        \end{proof}
       
       \begin{thm}\label{T-C03}
       Every $\delta_1$-primary element $Q$ of a multiplication $L$-module $M$ is 2-absorbing primary where $\delta_1$ is an expansion function on $M$. 
       \end{thm}        
       \begin{proof}
       Assume that a proper element $Q$ of a multiplication $L$-module $M$ is $\delta_1$-primary. Let $abN\leqslant Q$ for $a, b\in L$ and $N\in M$. Then as $Q$ is $\delta_1$-primary and $a(bN)\leqslant Q$, we have, either $bN\leqslant \delta_1(Q)$ or $aI_M\leqslant Q$ which implies either $bN\leqslant (\sqrt{Q:I_M})I_M$ or $ab\leqslant (Q:I_M)$. Thus $Q$ is a 2-absorbing primary element of $M$.
       \end{proof}
       
       Theorem \ref{T-C02} can also be acheieved through Corollary \ref{C-C01} and Theorem \ref{T-C03}.
       
       \begin{thm}\label{T-C10}
        Let $L$ be a PG-lattice and $M$ be a faithful multiplication PG-lattice $L$-module with $I_M$ compact. Let $\delta_1$ be an expansion function on $M$. If a proper element $Q$ of $M$ is  $\delta_1$-primary then $(Q:I_M)$ is a 2-absorbing primary element of $L$ and $\sqrt{Q:I_M}$ is a 2-absorbing element of $L$.
        \end{thm}
        \begin{proof}
        The proof follows either from Theorems \ref{T-C4}, \ref{T-C02} and \ref{T-C07} or from Theorems \ref{T-C03} and \ref{T-C07}.
        \end{proof}
           
         Note that, given an expansion function $\delta_1$ on a multiplication $L$-module $M$, $\delta_1(N)$ is an element of $M$ for all $N\in M$. Now we will see properties of this element $\delta_1(N)$.
        
         \begin{thm}\label{T-C14}
         Let $\delta_1$ be an expansion function on a multiplication $L$-module $M$.  Then for every prime element $N$ of  $M$, $(\sqrt{N:I_M})\delta_1(N)\leqslant N\leqslant \delta_1(N)$. 
         \end{thm}
         \begin{proof}
         Clearly, $N\leqslant \delta_1(N)$. As $N\in M$ is prime, by Theorem \ref{T-C04}, $N$ is a 2-absorbing element of $M$. Then from Corollary 2.13 in \cite{MB1}, $(N:I_M)$ is a 2-absorbing element of $L$. So by Lemma 2(iii) of \cite{JTY}, we have, $(\sqrt{N:I_M})^2\leqslant (N:I_M)$ which implies $(\sqrt{N:I_M})(\sqrt{N:I_M})I_M\leqslant (N:I_M)I_M$. Thus $(\sqrt{N:I_M})\delta_1(N) \leqslant N$ as $M$ is a multiplication module. 
         \end{proof} 
         
         \begin{thm}
         Let a proper element $N$ of a multiplication $L$-module $M$ be $p$-primary and 2-absorbing. Then $p ~ \delta_1(N)\leqslant N\leqslant \delta_1(N)$ where $\delta_1$ is an expansion function on $M$.                  
         \end{thm}
         \begin{proof}
         The proof runs on the same lines as that of the proof of above Theorem \ref{T-C14}.
         \end{proof}
         
         \begin{thm}
         For every proper element $N$ of a multiplication $L$-module $M$, $N\leqslant \delta_1((N:I_M)N)$ and equality holds if $N\in M$ is prime where $\delta_1$ is an expansion function on $M$.                  
         \end{thm}
         \begin{proof}
         Since $M$ is a multiplication lattice $L$-module, $N=(N:I_M)I_M$. As $(N:I_M)^2I_M=(N:I_M)N$, we have $(N:I_M)\leqslant \sqrt{((N:I_M)N):I_M}$ which implies $N=(N:I_M)I_M\leqslant (\sqrt{((N:I_M)N):I_M})I_M=\delta_1((N:I_M)N)$. Now assume that $N\in M$ is prime. Then by Proposition 3.6 of \cite{A}, $(N:I_M)\in L$ is prime. Let $r\in L$ be such that $r\leqslant \sqrt{((N:I_M)N):I_M}$. Then $r^nI_M\leqslant (N:I_M)N\leqslant N$ for some $n\in Z_+$. Since $N$ is prime and $r^n\leqslant (N:I_M)$, it follows that $r\leqslant (N:I_M)$. Thus $\sqrt{((N:I_M)N):I_M}\leqslant (N:I_M)$. Hence   $\delta_1((N:I_M)N)\leqslant N$. Therefore $N=\delta_1((N:I_M)N)$ if $N\in M$ is prime.
         \end{proof}
         
        \begin{thm}\label{T-C13}
         Let $L$ be a PG-lattice and $M$ be a faithful multiplication PG-lattice $L$-module with $I_M$ compact. Then given an expansion function $\delta_1$ on $M$,   $((\delta_1(N)):I_M)=\sqrt{N:I_M}$ for every proper element $N\in M$.
        \end{thm}
        \begin{proof}
        Since $M$ is a multiplication $L$-module, $\delta_1(N)=((\delta_1(N)):I_M)I_M$ and hence $((\delta_1(N)):I_M)I_M=(\sqrt{N:I_M})I_M$. Since $I_M$ is compact, by Theorem 5 of \cite{CT}, we have, $((\delta_1(N)):I_M)=\sqrt{N:I_M}$.
        \end{proof}

                \begin{l1}\label{L-C93}
               For every $q_i\in L\ (i\in Z_+)$ we have $\underset{i\in Z_+}{\wedge}\sqrt{q_i}=\sqrt{\underset{i\in Z_+}{\wedge}q_i}$.
               \end{l1}
               \begin{proof}
               The proof is obvious.
               \end{proof}  
                
                \begin{thm}
                Let $L$ be a PG-lattice and $M$ be a faithful multiplication PG-lattice $L$ module with $I_M$ compact.  Then given an expansion function $\delta_1$ on $M$,   $\underset{i\in Z_+}{\wedge}(\delta_1(N_i))=\delta_1({\underset{i\in  Z_+}{\wedge}N_i)}$ where $\{N_i\in M \mid i\in Z_+\}$.
               \end{thm}
                \begin{proof}
                By Lemma $\ref{L-C93}$ and Lemma $\ref{L-C91}$, we have $\delta_1({\underset{i\in  Z_+}{\wedge}N_i)}=(\sqrt{({\underset{i\in  Z_+}{\wedge}N_i)}:I_M}
                \ )I_M=$\\$(\sqrt{{\underset{i\in  Z_+}{\wedge}}(N_i:I_M)}\ )I_M=$$(\underset{i\in Z_+}{\wedge}\sqrt{N_i:I_M})I_M$=$\underset{i\in Z_+}{\wedge}((\sqrt{N_i:I_M})I_M)=\underset{i\in Z_+}{\wedge}(\delta_1(N_i))$.
                \end{proof}
        
        Now we relate the element $\delta_1(N)\in M$ with $rad(N)\in M$,the radical element of $M$ where $N\in M$. According to definition 3.1 in \cite{MB1}, the radical of a proper element $N$ in an $L$ module $M$ is defined as $\wedge \{P\in M \mid$ $P$ is a prime element and $N\leqslant P \}$ and is denoted as $rad(N)$. 
        
        \begin{thm}
         Given an expansion function $\delta_1$ on a multiplication $L$-module $M$, $\delta_1(N)\leqslant rad(N)$ for every proper element $N\in M$.
        \end{thm}
        \begin{proof}
         The proof follows from Lemma 3.5 in \cite{MB1}.
        \end{proof}
        
        \begin{thm}\label{T-C11}
        Let $L$ be a PG-lattice and $M$ be a faithful multiplication PG-lattice $L$-module with $I_M$ compact. Then given an expansion function $\delta_1$ on $M$,   $\delta_1(N)=rad(N)$ for every proper element $N\in M$.
        \end{thm}
        \begin{proof}
         The proof follows from Theorem 3.6 in \cite{MB1}.
        \end{proof}

        \begin{thm}\label{T-C12}
        Let $L$ be a PG-lattice and $M$ be a faithful multiplication PG-lattice $L$-module with $I_M$ compact. Let $\delta_1$ be an expansion function on $M$. If a proper element $N\in M$ is 2-absorbing then $\delta_1(N)$ is a 2-absorbing element of $M$ and hence a $(3,2)$-absorbing element of $M$. Further, $\delta_1(N)$ is a 2-absorbing primary element of $M$.
        \end{thm}
        \begin{proof}
        The proof follows from Theorem 3.9, Theorem 2.18 in \cite{MB1} and above Theorems \ref{T-C11}, ~\ref{T-C01}.
        \end{proof} 
        
        \begin{thm}
        Let $L$ be a PG-lattice and $M$ be a faithful multiplication PG-lattice $L$-module with $I_M$ compact. If a proper element $N\in M$ is  $2$-absorbing then one of the following statements hold true:
        \begin{enumerate}
        \item $\delta_1(N)=pI_M$ is a prime element of $M$ such that $p^2I_M\leqslant N$.
        \item $\delta_1(N)=p_1I_M\wedge p_2I_M$ and $(p_1p_2)I_M\leqslant N$ where $p_1I_M$ and $p_2I_M$ are the only distinct prime elements of $M$ that are minimal over $N$.
        \end{enumerate}
        \end{thm} 
        \begin{proof}
        The proof follows from Theorem 3.10 in \cite{MB1} and above Theorem \ref{T-C11}. 
        \end{proof}      
        
        \begin{thm}
         Let $L$ be a PG-lattice and $M$ be a faithful multiplication PG-lattice $L$-module with $I_M$ compact. Let $\delta_1$ be an expansion function on $M$. Let $\{N_i\in M \mid i\in Z_+\}$ be 2-absorbing elements of $M$.  If $\{\delta_1(N_i)\in M \mid i\in Z_+\}$ is (ascending or descending) chain of elements of $M$ then \\               
         \textcircled{1} $\underset{i\in Z_+}{\wedge}(\delta_1(N_i))$ is a $2$-absorbing and hence a 2-absorbing primary element of $M$.\\ \textcircled{2} $\underset{i\in Z_+}{\vee}(\delta_1(N_i))$ is a $2$-absorbing and hence a 2-absorbing primary element of $M$. 
        \end{thm}
        \begin{proof}
        The proof follows from Theorem 2.9 in \cite{MB1} and above Theorem \ref{T-C12}. 
        \end{proof}
        
        \begin{thm}
                Let $L$ be a PG-lattice and $M$ be a faithful multiplication PG-lattice $L$-module with $I_M$ compact. Let $\delta_1$ be an expansion function on $M$.
                If $N$ is a proper element of an $L$-module $M$ then $(\sqrt{N:K})K\leqslant \delta_1(N)$ for every proper element $K\in M$ such that $K\nleqslant N$.
                \end{thm}
                \begin{proof}
                Let $K\in M$ be a proper element such that $K\nleqslant N$. So $(N:K)\neq 1$. Let $P\in M$ be prime such that $N\leqslant P$. Then $(P:K)\in L$ is prime such that $\sqrt{N:K}\leqslant (P:K)$ which implies $(\sqrt{N:K})K\leqslant (P:K)K\leqslant P$. Thus whenever $P\in M$ is prime such that $N\leqslant P$, we have $(\sqrt{N:K})K\leqslant P$. It follows that $(\sqrt{N:K})K\leqslant rad(N)$ which implies $(\sqrt{N:K})K\leqslant \delta_1(N)$ by Theorem \ref{T-C11}.
                \end{proof}            
        
      \section{Expansion function $\delta_L$ on $L$ and $\delta_L$-primary element of $M$}
       
       We have studied expansion function on $L$ in \cite{MB}.
       Now given an expansion function $\delta_L$ on $L$, we define $\delta_L$-primary element of an $L$-module $M$.
       
       \begin{d1}
       Given an expansion function $\delta_L$ on a multiplicative lattice $L$, a proper element $P$ of an $L$-module $M$ is called {\bf $\delta_L$-primary} if for all $A\in M$ and $a\in L$, $aA\leqslant P$ implies either $A\leqslant P$ or $a\leqslant \delta_L(P:I_M)$.
       \end{d1}
       
       Now we give some characterizations of $\delta_L$-primary element of an $L$-module $M$.
               
               \begin{thm}
               Let $M$ be a CG-lattice $L$-module, $N\in M$ be a proper element and $\delta_L$ be an expansion function on $L$. Then the following statements are equivalent:-
                \begin{enumerate}
                 \item $N$ is a $\delta_L$-primary element of $M$.
                 \item for every $r\in L$ such that $r\nleqslant \delta_L(N:I_M)$, we have $(N:r)=N$.
                 \item for every $r\in L_\ast$, $A\in M_\ast$ , if $rA\leqslant N$ then either $A\leqslant N$ or $r\leqslant \delta_L(N:I_M)$.
                 \end{enumerate} 
                 \end{thm}
                \begin{proof}
                $(1)\Longrightarrow (2)$. Suppose $(1)$ holds. Let $Q\leqslant (N:r)$ be such that $r\nleqslant \delta_L(N:I_M)$ for $Q\in M, r\in L$. Then $rQ\leqslant N$. As $N$ is a $\delta_L$-primary element of $M$ and  $r\nleqslant \delta_L(N:I_M)$, we have, $Q\leqslant N$ and thus $(N:r)\leqslant N$. Since $rN\leqslant N$ gives $N\leqslant (N:r)$, it follows that $(N:r)=N$.\\
                 $(2)\Longrightarrow (3)$. Suppose $(2)$ holds. Let $rA\leqslant N$ and $r\nleqslant \delta_L(N:I_M)$ for  $r\in L_\ast$, $A\in M_\ast$. Then by $(2)$, we have, $(N:r)=N$. So $rA\leqslant N$ gives $A\leqslant(N:r)=N$.\\
                  $(3)\Longrightarrow (1)$. Suppose $(3)$ holds. Let $aQ\leqslant N$ and $Q\nleqslant N$ for $Q\in M, a\in L$.  As $L$ and $M$ are compactly generated, there exist $x'\in L_\ast$ and $Y,Y'\in M_\ast$ such that $x'\leqslant a, Y\leqslant Q, Y'\leqslant Q$ and $Y'\nleqslant N$. Let $x\in L_\ast$ such that $x\leqslant a$. Then $(x\vee x')\in L_\ast$, $(Y\vee Y')\in M_\ast$ such that $(x\vee x')(Y\vee Y')\leqslant aQ\leqslant N$ and $(Y\vee Y')\nleqslant N$. So by $(3)$, $(x\vee x')\leqslant \delta_L(N:I_M)$ which implies $x\leqslant \delta_L(N:I_M)$. Thus $a\leqslant \delta_L(N:I_M)$ and hence $N$ is a $\delta_L$-primary element of $M$.      
               \end{proof}
               
               \begin{thm}
               Let $M$ be a CG-lattice $L$-module, $N\in M$ be a proper element and $\delta_L$ be an expansion function on $L$. Then the following statements are equivalent:-
                \begin{enumerate}
                \item $N$ is a $\delta_L$-primary element of $M$.
                \item for every $A\in M$ such that $A\nleqslant N$, we have $(N:A)\leqslant \delta_L(N:I_M)$.
                \item for every $r\in L_\ast$, $A\in M_\ast$ , if $rA\leqslant N$ then either $A\leqslant N$ or $r\leqslant \delta_L(N:I_M)$.
                \end{enumerate} 
                \end{thm}
                \begin{proof}
                $(1)\Longrightarrow (2)$. Suppose $(1)$ holds. Let $A\in M$ be such that $A\nleqslant N$. Let $a\in L_\ast$ be such that $a\leqslant (N:A)$. Then $aA\leqslant N$. As $N$ is a $\delta_L$-primary element of $M$ and  $A\nleqslant N$, we have, $a\leqslant \delta_L(N:I_M)$ and thus $(N:A)\leqslant \delta_L(N:I_M)$.\\
                 $(2)\Longrightarrow (3)$. Suppose $(2)$ holds. Let $rA\in N$ and $A\nleqslant N$ for $r\in L_\ast$, $A\in M_\ast$. Then by $(2)$, we have,  $(N:A)\leqslant \delta_L(N:I_M)$. So $rA\leqslant N$ gives $r\leqslant (N:A)\leqslant \delta_L(N:I_M)$.\\
                  $(3)\Longrightarrow (1)$. Suppose $(3)$ holds. Let $aQ\leqslant N$ and $a\nleqslant \delta_L(N:I_M)$ for $Q\in M, a\in L$.  As $L$ and $M$ are compactly generated, there exist $x', x\in L_\ast$ and $Y\in M_\ast$ such that $x\leqslant a, x'\leqslant a, Y'\leqslant Q$ and $x'\nleqslant \delta_L(N:I_M)$. Let $Y\in M_\ast$ such that $Y\leqslant Q$. Then $(x\vee x')\in L_\ast$, $(Y\vee Y')\in M_\ast$ such that $(x\vee x')(Y\vee Y')\leqslant aQ\leqslant N$ and $(x\vee x')\nleqslant \delta_L(N:I_M)$. So by $(3)$, $(Y\vee Y')\leqslant N$ which implies $Y\leqslant N$. Thus $Q\leqslant N$ and hence $N$ is a $\delta_L$-primary element of $M$.  
                \end{proof} 
       
     In view of Theorems \ref{T-C41},~\ref{T-C4}, \ref{T-C2},~ \ref{T-C90},~\ref{T-C91} and \ref{T-C92}, we have the following results.  
       
        \begin{thm}
         Given an expansion function $\delta_L$ on $L$, a proper element $P$ of an $L$-module $M$ is $(\delta_0)_L$-primary if and only if it is prime.
         \end{thm}
         \begin{proof}
         The proof is obvious.
         \end{proof}
                  
         \begin{thm}
         Given an expansion function $\delta_L$ on $L$, a proper element $P$ of an $L$-module $M$ is $(\delta_1)_L$-primary if and only if it is primary.
        \end{thm} 
        \begin{proof}
        The proof is obvious.
        \end{proof}
        
        The above two results characterize the $(\delta_0)_L$-primary and $(\delta_1)_L$-primary elements of an $L$-module $M$ and relate them with prime and primary elements of $M$. Clearly, every $(\delta_0)_L$-primary element of a multiplication lattice $L$-module $M$ is $(\delta_1)_L$-primary. But the converse is not true which is clear from the example \ref{E-C1}. It is easy to see that the element $(4)$ is $(\delta_1)_L$-primary but not $(\delta_0)_L$-primary. 
        
       In view of Theorems \ref{T-C08}, \ref{T-C09}, \ref{T-C03} and \ref{T-C10}, we have following results. 
        
         \begin{thm}
         Every $(\delta_0)_L$-primary element $Q$ of an $L$-module $M$ is primary and 2-absorbing where $(\delta_0)_L$ is an expansion function on $L$. Also then both, $\sqrt{Q:I_M}$ and $(Q:I_M)$ are 2-absorbing and hence 2-absorbing primary elements of $L$. 
        \end{thm}  
        \begin{proof}
        The proof runs on the same lines as that of the proof of Theorem \ref{T-C08}.
        \end{proof}
                      
       \begin{thm}
       Every $(\delta_0)_L$-primary element $Q$ of a multiplication $L$-module $M$ is 2-absorbing primary where $(\delta_0)_L$ is an expansion function on $L$. Also then, $(Q:I_M)$ is a 2-absorbing primary element of $L$ provided $M$ is a faithful multiplication PG-lattice with $I_M$ compact and $L$ as a PG-lattice. Further, $\sqrt{Q:I_M}$ is a 2-absorbing element of $L$.
       \end{thm}  
       \begin{proof}
       The proof runs on the same lines as that of the proof of Theorem \ref{T-C09}.
       \end{proof}
               
               \begin{thm}
               Every $(\delta_1)_L$-primary element $Q$ of a multiplication $L$-module $M$ is 2-absorbing primary where $(\delta_1)_L$ is an expansion function on $L$. 
               \end{thm}        
               \begin{proof}
               The proof runs on the same lines as that of the proof of Theorem \ref{T-C03}.
               \end{proof}
               
               \begin{thm}
                Let $L$ be a PG-lattice and $M$ be a faithful multiplication PG-lattice $L$-module with $I_M$ compact. Let $(\delta_1)_L$ be an expansion function on $L$. If a proper element $Q$ of $M$ is  $(\delta_1)_L$-primary then $(Q:I_M)$ is a 2-absorbing primary element of $L$ and $\sqrt{Q:I_M}$ is a 2-absorbing element of $L$.
                \end{thm}
                \begin{proof}
                The proof runs on the same lines as that of the proof of Theorem \ref{T-C10}.
                \end{proof}
       In view of Theorems \ref{T-C2}, \ref{T-C90}, \ref{T-C91} and \ref{T-C92}, we have following results.             
        
       \begin{thm}
        If $\delta_L$ and $\gamma_L$ are expansion functions on $L$  such that $\delta_L(N:I_M)\leqslant\gamma_L(N:I_M)$ for each proper element $N$ of an $L$-module $M$ then every $\delta_L$-primary element of $M$ is $\gamma_L$-primary. In particular, if $N$ is a prime element of an $L$-module $M$ then $(N:I_M)\in L$ is a $\delta_L$-primary element for every expansion function $\delta_L$ on $L$.
       \end{thm}
       \begin{proof}
       The proof runs on the same lines as that of the proof of Theorem \ref{T-C2}.
       \end{proof}
       
       \begin{thm}
       Let $\delta_L$ be an expansion functions on $L$. If a proper element $P$ of an $L$-module $M$ is $\delta_L$-primary then\\       
       \textcircled{1} $(P:a)=P$ for all $a\in L$ such that $a\nleqslant\delta_L(P:I_M)$.\\               
       \textcircled{2} $(P:q)$ is a $\delta_L$-primary element of $M$ for all $q\in L$. 
       \end{thm}
       \begin{proof}
       The proof runs on the same lines as that of the proof of Theorem \ref{T-C90}.
       \end{proof}
       
       \begin{thm}
        Let $\delta_L$ be an expansion function on $L$. If $\{P_i\mid i\in \triangle\}$ is a chain of $\delta_L$-primary elements of an $L$-module $M$ with $I_M$ compact then the element $P=\underset{i\in \triangle}{\vee}P_i$ is also a $\delta_L$-primary element of $M$.
       \end{thm}
       \begin{proof}
       The proof runs on the same lines as that of the proof of Theorem \ref{T-C91}.
       \end{proof}
       
       According to \cite{MB}, an expansion function $\delta_L$ on $L$ is said to have meet preserving property, if for all $a, b\in L$, $\delta_L(a\wedge b)=\delta_L(a)\wedge \delta_L(b)$.
       
       \begin{thm}
        Let the expansion function $\delta_L$ on $L$ have the meet preserving property. If $Q_1,\ Q_2,\cdots,\ Q_n$ are $\delta_L$-primary elements of an $L$-module $M$ and $r=\delta_L(Q_i:I_M)$ for all $i=1,\ 2,\cdots,\ n$ then the element $Q=\bigwedge_{i=1}^{n}Q_i$ is also a $\delta_L$-primary element of $M$.
        \end{thm}
        \begin{proof}
        The proof runs on the same lines as that of the proof of Theorem \ref{T-C92}.
        \end{proof}
        
        We conclude this paper with the following result which shows that if an element in $M$ (or $L$) is $\delta_L$-primary then its corresponding element in $L$ (or $M$) is also $\delta_L$-primary.
               
               \begin{thm}\label{T-C1}
               Let $L$ be a PG-lattice and $M$ be a faithful multiplication PG-lattice $L$-module with $I_M$ compact. Let $N$ be a proper element of an $L$-module $M$. Given an expansion function $\delta_L$ on $L$, the following statements are equivalent:
               \begin{enumerate}
               \item $N$ is a $\delta_L$-primary element of $M$.
               \item $(N:I_M)$ is a $\delta_L$-primary element of $L$.
               \item $N=qI_M$ for some $\delta_L$-primary element $q\in L$.
               \end{enumerate} 
               \end{thm}
               \begin{proof}
              $(1)\Longrightarrow (2)$. Assume $N$ is a $\delta_L$-primary element of $M$. Let $ab\leqslant (N:I_M)$ for $a, b\in L$. Then $a(bI_M)\leqslant N$. As $N$ is a $\delta_L$-primary element of $M$, we have, either $bI_M\leqslant N$ or $a\leqslant \delta_L(N:I_M)$ which implies either $b\leqslant (N:I_M)$ or $a\leqslant \delta_L(N:I_M)$ and thus $(N:I_M)$ is a $\delta_L$-primary element of $L$.\\
               $(2)\Longrightarrow (1)$. Assume $(N:I_M)$ is a $\delta_L$-primary element of $L$. Let $aA\leqslant N$ and $a\nleqslant \delta_L(N:I_M)$ for $a\in L, A\in M$. As $M$ is a multiplication $L$-module, we have, $A=xI_M$ for some $x\in L$. So $aA\leqslant N$ implies $ax\leqslant (N:I_M)$. Since $(N:I_M)$ is a $\delta_L$-primary element of $L$ and $a\nleqslant \delta_L(N:I_M)$, we must have  
               $x\leqslant (N:I_M)$ which implies $A\leqslant N$ and thus $N$ is a $\delta_L$-primary element of $M$.\\
               $(2)\Longrightarrow (3)$. Suppose $(N:I_M)$ is a  $\delta_L$-primary element of $L$. Since $M$ is a multiplication lattice $L$-module, by Proposition 3 of \cite{CT}, we have, $N=(N:I_M)I_M$ and hence $(3)$ holds.\\
                $(3)\Longrightarrow (2)$. Suppose $N=qI_M$ for some $\delta_L$-primary element $q\in L$. As $M$ is a multiplication lattice $L$-module, by Proposition 3 of \cite{CT}, we have $N=(N:I_M)I_M$. Since $I_M$ is compact, $(2)$ holds by Theorem 5 of \cite{CT}.
               \end{proof}
               
        In view of above Theorem \ref{T-C1}, we give the following corollary without proof.
               
               \begin{c1}
               Let $\delta_L$ be an expansion function on $L$. If a proper element $N$ of an $L$-module $M$ is $\delta_L$-primary then $(N:I_M)$ is a $\delta_L$-primary element of $L$. Converse holds if $M$ is a multiplication lattice $L$-module.
               \end{c1}
               
      Thus, {\bf a proper element $N$ of a multiplication lattice $L$-module $M$ is a $\delta_L$-primary element of $M$ if and only if  $(N:I_M)$ is  $\delta_L$-primary element of $L$.}

\end{document}